\RequirePackage{fix-cm}
\documentclass[smallextended]{svjour3}       
\smartqed  
\usepackage{graphicx}
\usepackage{amsmath,amssymb}
\usepackage[round]{natbib}
\usepackage[utf8]{inputenc}
\usepackage{CJK}

\renewcommand{\Re}{\operatorname{Re}} 
\newcommand{\R}{\mathbb{R}} 
\newcommand{\C}{\mathbb{C}} 
\newcommand{\ord}{\mathcal O} 
\newcommand{\cc}[1]{\overline{#1}} 
\newcommand{\del}[1]{\partial_{#1}} 
\newcommand{\til}[1]{\tilde #1} 
\newcommand{\ev}{\xi} 
\renewcommand{\prod}[2]{\left\langle #1,#2\right\rangle} 
\newcommand{\sgn}{\operatorname{sgn}} 
\newcommand{\zhong}{\begin{CJK}{UTF8}{bsmi}中 \end{CJK}}
\newtheorem{mylemma}{Lemma}[section] 

\DeclareMathOperator{\interior}{int}  

\journalname{Journal of Mathematical Biology}
\begin{document}

\title{Global Hopf Bifurcation in the ZIP regulatory system \thanks{This work was supported by: German Research Foundation grant number CH 958/1-1, Excellence Initiative II Heidelberg University ``Mobilit\"atsma\ss nahmen im Rahmen internationaler Forschungskooperationen 2013-2014'' project number D.80100/13.009, and EPSRC grant number EP/K036521/1.}
}


\author{Juliane Claus  \and
        Mariya Ptashnyk \and
        Ansgar Bohmann \and
	Andr\'es Chavarr\'ia-Krauser  
}

\authorrunning{Claus, Ptashnyk, Bohmann, Chavarr\'ia-Krauser} 

\offprints{A. Chavarr\'ia-Krauser}

\institute{J. Claus \at
              Interdisciplinary Center for Scientific Computing, Heidelberg University\\
              \email{juliane.claus@bioquant.uni-heidelberg.de}           
           \and
	M. Ptashnyk \at
	  Division of Mathematics, University of Dundee\\
	     \email{mptashnyk@maths.dundee.ac.uk}    
	\and
	A. Bohmann \at
              Interdisciplinary Center for Scientific Computing, Heidelberg University
        \and
          A. Chavarr\'ia-Krauser \at
              Interdisciplinary Center for Scientific Computing, Heidelberg University\\
	\email{andres.chavarria@bioquant.uni-heidelberg.de}
}

\date{Received: date / Accepted: date}

\maketitle

\begin{abstract}
Regulation of zinc uptake in roots of \emph{Arabidopsis thaliana} has recently been modeled by a system of ordinary differential equations based on the uptake of zinc, expression of a transporter protein and the interaction between an activator and inhibitor. For certain parameter choices the steady state of this model becomes unstable upon variation in the external zinc concentration. Numerical results show periodic orbits emerging between two critical values of the external zinc concentration. Here we show the existence of a global Hopf bifurcation with a continuous family of stable periodic orbits between two Hopf bifurcation points. The stability of the orbits in a neighborhood of the bifurcation points is analyzed by deriving the normal form, while the stability of the orbits in the global continuation is shown by calculation of the Floquet multipliers. From a biological point of view, stable periodic orbits lead to potentially toxic zinc peaks in plant cells. Buffering is believed to be an 
efficient way to deal with strong transient variations in zinc supply. We extend the model by a buffer reaction and analyze the stability of the steady state in dependence of the properties of this reaction. We find that a large enough equilibrium constant of the buffering reaction stabilizes the steady state and prevents the development of oscillations. Hence, our results suggest that buffering has a key role in the dynamics of zinc homeostasis in plant cells. 
\end{abstract}

\section{Introduction}
\label{intro}

Zinc is an essential micronutrient for plants, because it plays an important role in many enzymes catalyzing vital cellular reactions. In higher doses, however, zinc is toxic. Therefore, plants have to strictly control and adjust the uptake of zinc through their roots depending on its concentration in the surrounding soil. This is achieved by a complicated control system consisting of sensors, transmitters and zinc transporter proteins.

The main influx transporters of zinc are proteins of the ZIP (zinc-, iron-responsive proteins) family \citep{Grotz_1998,Guerinot_2000}. They were shown to be highly expressed under conditions of zinc deficiency and down-regulated under zinc excess \citep{Talke_2006}. The mechanism of this regulation has been studied by \citet{Assuncao_2010}. They showed that \emph{ZIP4} in \emph{Arabidopsis thaliana} is regulated by the transcription factors bZIP19 and bZIP23 of the basic-region leucine zipper (bZIP) family. These transcription factors apparently induce ZIP expression by binding to a zinc deficiency response element (ZDRE), which is present in the promoter region of most ZIP proteins. Thereby, bZIPs increase ZIP transporter expression under zinc deficiency. When the internal zinc concentration in the cell rises, bZIPs are inhibited and the expression of zinc transporters decreases strongly. The mechanistic details of this inhibition are yet unclear, as bZIP transcription factors do not seem to 
have a zinc 
binding site for sensing the internal zinc concentration \citep{Assuncao_2010}. Most likely there are additional sensors that bind zinc and inhibit bZIP19 and bZIP23 \citep{Assuncao_2010_mini}. 

These biological observations gave rise to a mathematical model, which we developed and published previously in \citet{Claus_2012}. The interested reader finds more details on the model, notation, parameters and biological implications there. The model is a four-dimensional autonomous system of ordinary differential equations
\begin{equation}
  \label{eq:form}
\begin{aligned}
	\frac{du}{dt} = F(u,\mu) \; , & \quad t \in [0,\infty)\; ,\\
	u(0) = u^0 \; ,
\end{aligned}
\end{equation}
with a parameter $\mu \in \mathcal{M}:=[0,m]$, where $m>0$ sufficiently large, a $C^1$ function $F: \mathbb{R}^4 \times \mathcal{M} \to \mathbb{R}^4$, and non-negative initial conditions $u^0 \in \mathbb{R}^4_+$. Referring to the original notation of \citet{Claus_2012}, the components of $u$ are the gene expression level, internal zinc concentration, the activation and inhibition levels, and the parameter $\mu$ represents the external zinc concentration. In this paper we study the qualitative behavior of the solutions with respect to the parameter $\mu$.  

Dynamical systems depending on a parameter $\mu$ can undergo bifurcations that change the qualitative behavior of solutions substantially when crossing a critical value   $\mu^*$. In the Hopf (or Poincar\'e-Andronov-Hopf) bifurcation, first described by \citet{Hopf_1942}, a steady state changes stability as a pair of complex conjugate eigenvalues of the linearization cross the imaginary axis and a family of periodic orbits bifurcates from the steady state.

While the original results by \citet{Hopf_1942} give conditions for the local existence of such orbits, later works are focused on the global continuation of periodic solutions. The first result on global continuability of periodic orbits  for ordinary differential equations was presented in \citet{Alexander_1978} using methods of algebraic topology. Another proof of the global result was given in \citet{Ize_1976} using homotopy theory. The Fuller index, an index for periodic solutions of a system of  autonomous  equations, was used in \citet{Chow1978} to generalize 
the global Hopf bifurcation theorem, proved  by  \citet{Alexander_1978},  to functional differential equations. 
An \emph{orbit index} and a \emph{center index} were introduced in \citet{Mallet_1982} and applied also in \citet{Alligood1983} to analyze the large connected sets of periodic solutions of a one-parameter differential equation. Hopf bifurcation points with a center index of $1$ are called \emph{sources}, and those with center index of $-1$ are called \emph{sinks}. \citet{Mallet_1982} showed, that if a set of orbits is bounded with respect to the parameter, solution, and periods of the orbits, then the set  must have as many source as sink Hopf bifurcations. Each source is connected to a sink by an oriented one-parameter path of orbits that contains no orbits with zero orbit index.
The results on global continuability  for locally continuable non-M\"obius orbits  of general $C^1$ dynamical systems were obtained in \citet{Alligood1984}. The term ``global'' Hopf bifurcation is sometimes used to denote only unbounded sets of periodic solutions \citep{Fiedler_1986}. We, however, will use it in the sense of \citet{Alexander_1978} to mean any non-local continuation of a family of periodic orbits.

The paper is organized as follows. In Sect.~\ref{sec:model} we consider  the  well-posedness and uniform boundedness of the solutions, and prove the existence of a unique stationary solution of the model for the regulation of zinc uptake. Then,
we shall show in Sect.~\ref{sec:local} that for two critical parameter values $\mu_{1}$ and $\mu_2$  Hopf bifurcations occur in the model. By deriving the normal form in Sect.~\ref{sec:normal}, we analyze the stability of the periodic orbits in a neighborhood of the bifurcation points. Furthermore, in Sect.~\ref{sec:global} we show the existence of a global continuous path of periodic orbits between the two Hopf bifurcation points and in Sect.~\ref{sec:floquet} analyze the stability of periodic solutions  by calculating the Floquet multipliers via the monodromy matrix. Finally, in Sect.~\ref{buffering},  we extend the model by a buffer reaction and analyze the properties needed to ensure the  stability of the stationary solutions.

\section{Mathematical model}
\label{sec:model}

In \citet{Claus_2012} we studied mathematical models  for  the regulation of ZIP transporters in response to external zinc concentration. The most promising model identified in this study included a dimerizing activator  $A$, as well as an inhibitor $I$ that senses the internal zinc concentration  $Z$ and inhibits the activator. Gene activity (denoted by $G$) is induced by the activator and leads to transcription of mRNA ($M$) and production of ZIP transporter proteins ($T$).

We consider a simplification of the original system, where transcription and translation of transporter proteins are assumed to be quasi-stationary and in equilibrium  $G=M=T$. Thereby, the system reduces to four equations describing the time evolution of $G$, $Z$, $A$, and $I$, now denoted by $u_1$, $u_2$, $u_3$, and $u_4$, respectively. The non-dimensionalized model reads
\begin{equation}
 \label{eq:model}
\begin{aligned}
	\frac{du_1}{dt} &= \kappa u_3^2(1-u_1) - u_1\; ,\\
	\frac{du_2}{dt} &= u_1f(\mu) - u_2\; ,\\
	\frac{du_3}{dt} &= 1- \gamma_1 u_3 u_4 -  u_3\; ,\\
	\frac{du_4}{dt} &= \gamma_2 u_2 - \gamma_3 u_3 u_4 - (1+\gamma_2 u_2) u_4\; ,
\end{aligned}
\end{equation}
where $f(\mu):=\frac{\mu}{\mu+ K}$ defines the influx of zinc dependent on the external zinc concentration $\mu$. 
The parameter  $\kappa$  describes the induction of gene activity ($u_1$) by the dimerized activator ($u_3$), $\gamma_1$  and $\gamma_3$ the binding between activator and inhibitor ($u_3$ and $u_4$), $\gamma_2$ the binding between zinc ($u_2$) and the inhibitor ($u_4$), and $K$ is the Michaelis constant of the reaction between external zinc ($\mu$) and zinc transporters ($u_1$). All  parameter values, here $K =13$, $\kappa =20$, $\gamma_1=380$, $\gamma_2= 1000$, $\gamma_3= 1672$,  have been estimated from literature and by fitting the model to experimental data \citep{Claus_2012}.
As a property of the soil, the external zinc concentration is considered as the bifurcation parameter $\mu\in \mathcal M$.

Considering $F(u,\mu)= (F_1(u), F_2(u,\mu), F_3(u), F_4(u))^T$, where
\begin{align*}
F_1(u)&= \kappa u_3^2(1-u_1) - u_1\; , &\quad F_2(u,\mu)&=u_1f(\mu) - u_2\; ,\\
F_3(u)&= 1- \gamma_1 u_3 u_4 - u_3\; , &\quad F_4(u)&=\gamma_2 u_2 - \gamma_3 u_3 u_4 - (1+\gamma_2 u_2) u_4\; ,
\end{align*}
we can rewrite the system \eqref{eq:model} in the form \eqref{eq:form}. $F(u,\mu)$ is  continuously differentiable in $u$  and its Jacobian is given by
\begin{equation}
  \label{eq:jac}
	J(u, \mu)=
	\begin{pmatrix}
	-\kappa u_3^2 - 1 & 0 & 2\kappa u_3(1-u_1) & 0 \\
	f(\mu) & -1 & 0 & 0 \\
	0 & 0 & -\gamma_1 u_4 - 1 & -\gamma_1 u_3\\
	0 & \gamma_2(1-u_4) & -\gamma_3 u_4 & -\gamma_3 u_3-1-\gamma_2 u_2
	\end{pmatrix}\;.
\end{equation}

Two theorems on the existence and uniqueness of the global solution and a positive steady state of system \eqref{eq:model} are given below. The interested reader finds their proofs in Appendices \ref{appendix:exitence} and \ref{appendix:ss}.

\begin{theorem}\label{existence}
For any initial value  $u^0 \in \mathcal S$, where $\mathcal S=[0,1]^4$,  there exists a unique  global solution  $u \in C^\infty([0, \infty)\times \mathcal M)$ of the system \eqref{eq:model} and $u(t,\mu) \in \mathcal S$ for all $t \in [0, \infty)$ and any $\mu\in\mathcal{M}$.
\end{theorem}

\begin{theorem}\label{ssunique}
For any positive parameter set $\kappa$, $K$, $\gamma_1$ to $\gamma_3$ and $\mu$, there exists a unique steady state $u^\ast(\mu)$ of the system \eqref{eq:model} in $\mathcal{S}=[0,1]^4$. Furthermore, $u^\ast \in C^\infty\big((0,\infty)\big)$. In particular, the steady state is isolated. 
\end{theorem}

\section{Local Hopf bifurcation}\label{sec:local}
In numerical simulations of the model  \eqref{eq:model} with sufficiently high values of the parameter $\gamma_1$ we observe a stable steady state for small $\mu$, which becomes unstable at some critical  parameter value  $\mu_1$ and stable again after some value $\mu_2>\mu_1$ \citep{Claus_2012}. Between these values, i.e. for $\mu \in (\mu_1, \mu_2)$, numerical simulations of the system  \eqref{eq:model} show the  existence of  stable limit cycles depending continuously on $\mu$.
In this section we show that at both critical values of $\mu$ a  Hopf bifurcation occurs. Since experiments have shown that external zinc concentrations above $30\,\mu$M are lethal for \emph{A. thaliana} \citep{Talke_2006}, from now on we will use the upper bound $m=30$ for $\mu$, i.e. $\mu \in \mathcal{M}=[0,30]$ [compare Eq. \eqref{eq:form}]. The model parameters are chosen as in Sect.~\ref{sec:model}.

\begin{theorem}\label{theorem1}
There exist two critical values $\mu_1, \mu_2 \in \mathcal{M}$ of the external zinc concentration  for which a Hopf bifurcation occurs in  the system \eqref{eq:model}. 
\end{theorem}

\begin{proof} 
We shall show that all criteria for existence of  a local Hopf bifurcation  \citep{Hopf_1942, Hassard_1981} are satisfied by the system  \eqref{eq:model}. The right hand side $F: \mathbb R^4\times \mathcal M  \to \mathbb R^4$ of   \eqref{eq:model} is  continuously differentiable in $u$  and $\mu$. Theorem \ref{ssunique} guarantees the existence of a unique equilibrium $u^\ast(\mu)$ inside the invariant set $\mathcal S$. Using Newton  and continuation methods  we find numerically an explicit family of unique positive stationary solutions $u^\ast(\mu)$ of \eqref{eq:model}. By Theorem \ref{ssunique} the equilibrium $u^\ast(\mu)$  is isolated. 

From the continuity of $f(\mu)=\frac{\mu}{\mu +K}$, of $u^\ast(\mu)$ with respect to  $\mu$, and of the determinant, we obtain that the eigenvalues $\lambda_i(\mu)$, $i=1,2,3,4$,  of the Jacobian $J(u^\ast(\mu), \mu)$ evaluated at the stationary solution,  vary continuously with $\mu$. 

A diagram of these eigenvalues  in the complex plane  for different values of $\mu$ is shown in Fig.~\ref{fig:eigen}.  
For $\mu=0$ the eigenvalues can be calculated explicitly and are equal to $\lambda_{1,2}=-1$, $\lambda_3=-\kappa-1=-21$ and $\lambda_4=-\gamma_3-1 = -1673$. For $\mu \in (0,30]$ the eigenvalues were computed numerically by a reduction of the Jacobian matrix to Hessenberg form followed by QR iteration as described in \citet{nr3}.
 
\begin{figure}
 \includegraphics[width=\textwidth]{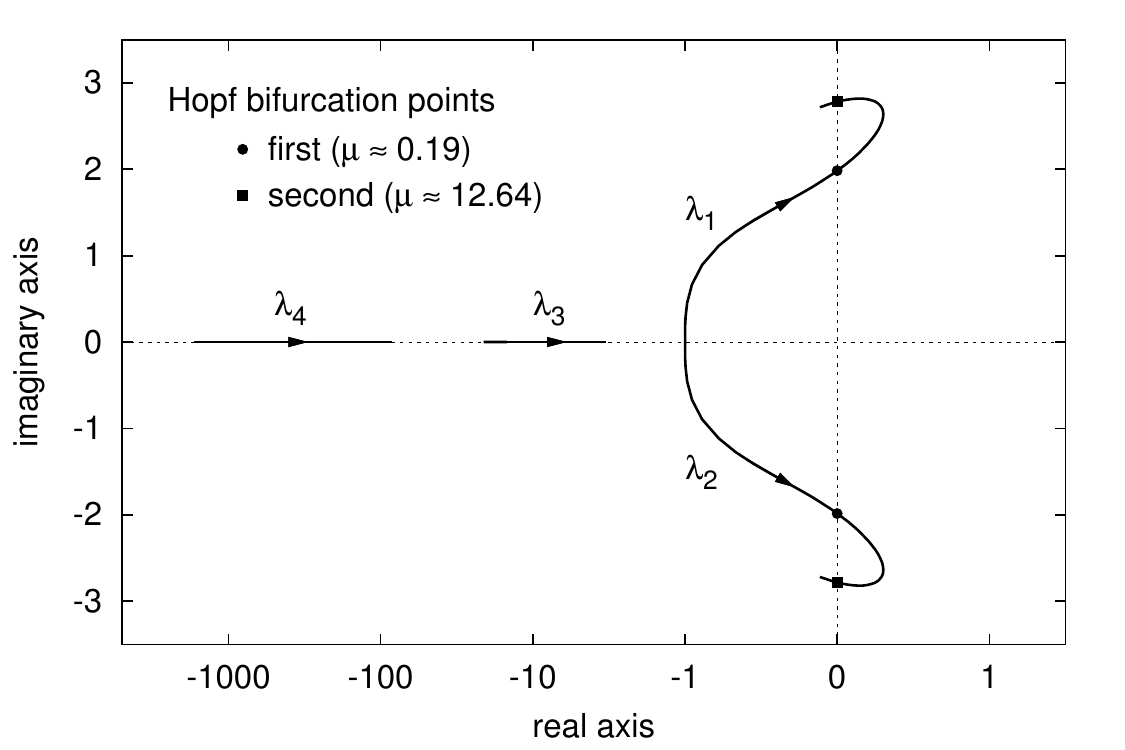}
\caption{Eigenvalues of the Jacobian matrix of F at the steady state for varying external zinc concentrations between $0$ and $30$. Arrows show the direction of increasing external zinc concentration. The eigenvalues $\lambda_3$ and $\lambda_4$ are real and negative for the entire range of external zinc concentrations, while the complex conjugate pair $\lambda_{1,2}$ crosses the imaginary axis twice. These points are marked as first and second Hopf bifurcation points.}
\label{fig:eigen}
\end{figure}

 The two smallest eigenvalues $\lambda_3$ and $\lambda_4$ have negative real parts for all  $\mu \in \mathcal{M}=[0, 30]$, while $\lambda_1$ and $\lambda_2$ form a complex conjugate pair that crosses the imaginary axis in two points (Fig.~\ref{fig:eigen}).
The values of $\lambda_1,\ldots,\lambda_4$ for the critical parameter points $\mu_1=0.189537\ldots$ and $\mu_2=12.6432\ldots$ at which  $J(u^\ast(\mu), \mu)$ exhibits two purely imaginary eigenvalues are shown in Table \ref{tab:critical}.  The purely imaginary eigenvalues $\lambda_{1,2}(\mu_i)$ are simple and none of the multiples $n \,\lambda_{1,2}(\mu_i)$, with $n \in \mathbb N\setminus\{1\}$, is an eigenvalue of $J(u^\ast(\mu_i), \mu_i)$ for $i=1,2$. This follows from $\operatorname{det} J(u^\ast(\mu_i), \mu_i)>0$ and $\operatorname{tr} J(u^\ast(\mu_i), \mu_i) < 0$ which imply $\lambda_{3,4}(\mu_i)<0$. The  eigenvalues $\lambda_{1, 2}(\mu)$  cross the imaginary axis with non-zero speed, i.e. $\frac{d}{d\mu} \Re \lambda_i (\mu_j)\neq 0$ for $i,j=1,2$, which was verified numerically (Table~\ref{tab:critical}).

A short analysis of the possible eigenvalue combinations is conducted in Appendix \ref{appendix:spectrum}. It is possible to show that all real eigenvalues are smaller than $-1$. Further, either the eigenvalues are: all real, or two are real and the other build a complex pair (case shown in Fig.~\ref{fig:eigen}), or two complex pairs are built, where one has a negative real part, while the sign of the other is indeterminate (numerical simulations suggest also a negative real part). Note that these results are general and valid for all positive parameters $p \in \mathcal{P}$.

\begin{table}
\caption{Critical values of $\mu$ with the corresponding eigenvalues $\lambda_1, \ldots, \lambda_4$ of the Jacobian matrix, derivative of the real part of the eigenvalue with respect to $\mu$ and sign of the real part of the parameter $b$ in the normal form.}
\label{tab:critical}
\begin{tabular}{llllllll}
  \hline\noalign{\smallskip}
	 &value of $\mu $  & $\lambda_{1,2}$ &  $  \lambda_3 $ & $\lambda_4$ & $\frac{d}{d\mu} \Re (\lambda_{1,2}) $ & $\sgn(\Re b)$\\
  \noalign{\smallskip}\hline\noalign{\smallskip}
	$\mu_1$ &$0.189\ldots$& $\pm i1.983\ldots$ & $ -3.474\ldots $ & $-238.6\ldots$ &   $1.83\ldots$ & -1\\
  \noalign{\smallskip}
	$\mu_2$ & $12.64\ldots$ & $\pm i2.782\ldots$ & $ -5.599\ldots$ & $ -84.43\ldots$  &   $-0.0126\ldots$ & -1\\
  \noalign{\smallskip}\hline
\end{tabular}
\end{table}

Therefore, based on the  Hopf Bifurcation Theorem  \citep{Hopf_1942, Hassard_1981} we conclude that the system  \eqref{eq:model} has  a  one-parameter family of periodic solutions in a neighborhood of both critical values  $\mu_{1}$ and $\mu_2$, bifurcating from the stationary solution. 
\end{proof}

\section{Derivation of the normal form and stability of Hopf bifurcation} \label{sec:normal}

In the last section we showed the existence of periodic orbits in the neighborhood of the two Hopf bifurcation points.  Further information is needed to determine the stability and direction of the periodic solutions.

\begin{theorem}\label{theorem2}
At both critical parameter values  $\mu_1$ and $\mu_2$ a supercritical Hopf bifurcation occurs for the system \eqref{eq:model} and the bifurcating periodic solutions are stable.  In the first critical point $\mu_1$ the orbits arise  for  increasing values  $\mu>\mu_1$, whereas  in the second critical point $\mu_2$ the orbits arise  for decreasing values $\mu<\mu_2$.
\end{theorem}

\begin{proof}
In order to apply the results known from literature, we  shift the critical parameter value and stationary solution  to zero by setting $\tilde \mu := \mu - \mu^\ast$,    $\tilde u := u - u^\ast(\tilde \mu+ \mu^\ast)$  and rewrite \eqref{eq:model} as
\begin{eqnarray}\label{eq:model_2}
\frac{d\tilde u}{dt} = F(\tilde u + u^\ast(\tilde \mu+ \mu^\ast), \tilde \mu + \mu^\ast) \; .
\end{eqnarray}
As $F$ is smooth, by \citet[Theorems 2.9, 3.3]{Haragus_2011} and  \citet{Ipsen_1998} and using the results of Theorem~\ref{theorem1},  we conclude that the system \eqref{eq:model_2} possesses a two-dimensional center manifold for sufficiently small $\tilde \mu$. 
Eq. \eqref{eq:model_2} reduced on this manifold is transformed by a specific polynomial transformation into the normal form \citep{Hassard_1981}
 \begin{equation}\label{normal_form}
 \frac{dA}{dt} = i\omega \,  A  +  a\, \tilde \mu\, A + b \, A \, |A|^2 + \ord\left(|A|\left(|\tilde \mu| + |A|^2\right)^2\right)   \; .
 \end{equation}
The solutions of \eqref{eq:model_2} on the center manifold are then of the form 
 \begin{equation}\label{center_manifold_u_normal}
 \tilde u = A \xi + \overline{A\xi} + \Psi( A, \bar A, \tilde \mu), \qquad A(t) \in \mathbb C  \; , 
 \end{equation}
with $\Psi(0,0, 0)=0$, $\del{A} \Psi(0,0, 0)=0$, and $\del{\cc{A}}\Psi(0,0, 0)=0$,  where $\xi $ is the eigenvector of $J(u^\ast, \mu^\ast)$ for the purely imaginary eigenvalue $\lambda_1= i\omega$. 
For $ \Psi( A, \bar A, \tilde \mu)$ a polynomial ansatz
\begin{align}\label{phi}
	\Psi(A,\cc{A}, \tilde \mu)=\sum \psi_{rsq} A^r\cc{A}^s \tilde \mu ^q,
\end{align}
with  $\psi_{100}=\psi_{010}=0$ and  $\psi_{rsq}=\cc{\psi}_{srq}$ is made \citep{Haragus_2011}.
Substituting the form \eqref{center_manifold_u_normal} for $u$ into equations \eqref{eq:model_2} we obtain 
 \begin{eqnarray}\label{transform_u_A}
 &&(\xi + \partial_A \Psi ) \frac{ dA}{ dt} +
   (\overline\xi + \partial_{\bar A} \Psi ) \frac{ d\bar{A}}{ dt} \\  
&&   \hspace{2 cm } =  F(A\xi +  \overline{A\xi} + \Psi (A, \bar A, {\tilde \mu})+ u^\ast(\tilde \mu+ \mu^\ast), \tilde \mu + \mu^\ast)\; . \nonumber
 \end{eqnarray}
Knowing the first term in the Taylor expansion of $F$ at $\tilde \mu =0$ and $\tilde u =0$ to be $F(u^\ast(\mu^\ast),\mu^\ast)=0$ we get
\begin{equation}\label{taylor}
\small
\begin{split}
	F(\tilde u + u^\ast(\tilde \mu+ \mu^\ast), \tilde \mu + \mu^\ast) &= D^{10}F \cdot \tilde u + D^{10}F\cdot \del{\mu}u^*\til{\mu} + D^{01}F\til{\mu}\\
		&\phantom{=} + \frac{1}{2} D^{20}F (\tilde u, \tilde u) + D^{11}F \cdot \tilde u \,  \tilde{\mu}
		 + \frac 12 D^{20}F (\del{\mu}u^\ast, \del{\mu}u^\ast) \til{\mu}^2\\
		 &\phantom{=} + \frac 12 D^{10}F \cdot \partial^2_\mu u^\ast\tilde{\mu}^2 + \frac{1}{2} D^{02}F \tilde{\mu}^2 + \ldots \; ,
\end{split}
\end{equation}
where $\tilde u = A\ev + \cc{A\ev} + \Psi( A, \bar A, \tilde \mu)$ and the derivatives are defined by
\begin{align*}
D^{pq}F &:=\frac{\partial^{p+q} F(u,\mu)}{\partial u^p \partial \mu^q}\bigg|_{(u^*(\mu^*),\mu^*)}.
\end{align*}
As before, we  write  $J=D^{10}F(u^\ast(\mu^\ast), \mu^\ast)$.  
Higher order derivatives in $u$ are multilinear forms  $D^{20}F:(\C^n)^2\rightarrow \C^n$ and $D^{30}F:(\C^n)^3\rightarrow \C^n$, with $n=4$. These forms are applied to $x, y, z \in \C^n$ to obtain vectors in $\C^n$ where the $i$-th components are given by
\begin{align*}
  (D^{20}F(x,y))_i &:= \sum_{j,k}^n \frac{\partial^2 F_i}{\partial u_j \partial u_k} x_j y_k \; ,\\
  (D^{30}F(x,y,z))_i &:= \sum_{j,k,l}^n \frac{\partial^3 F_i}{\partial u_j \partial u_k \partial u_l} x_j y_k z_l \; .
\end{align*}

For $F$ in the  system \eqref{eq:model} the second and third derivatives  $D^{20}F(x,y)$ and $D^{30}F(x,y,z)$, where $x,y,z\in\C^4$,  are given by
\begin{align*}
	D^{20}F(x,y)= 
	\begin{pmatrix}
	-2\kappa u^\ast_3(x_1y_3 + x_3y_1) + 2\kappa (1-u^\ast_1)x_3y_3\\
	0\\
	-\gamma_1 (x_3y_4 + x_4y_3)\\
	-\gamma_2 (x_2y_4 + x_4y_2) - \gamma_3(x_3y_4 + x_4y_3)
	\end{pmatrix}
\end{align*}
and
\begin{align*}
	D^{30}F(x,y,z)=
	\begin{pmatrix}
	-2 \kappa (x_1y_3z_3 + x_3y_1z_3 + x_3y_3z_1) \\ 0\\ 0\\ 0
	\end{pmatrix}.
\end{align*}

Using the Taylor expansion \eqref{taylor} of $F$ together with  the expression  \eqref{phi} for $\Psi$ and  \eqref{normal_form} in the equation \eqref{transform_u_A},  and comparing the coefficients we obtain 
\begin{align*}
	\ord(A):\quad 
	&& (i\omega-J) \ev &= 0\; , & (-i\omega-J) \cc{\ev} &= 0\; .
\end{align*}
This is  the eigenvalue problem for the Jacobian of $F$ in $(u^*(\mu^*),\mu^*)$ for purely imaginary eigenvalues $\pm i\omega$. 
Further on, we find:
\begin{align}
	\ord(\til{\mu}):\quad 
	& -J \psi_{001} = D^{01}F + D^{10}F \cdot \del{\mu}u^*  \label{eq_psi001} \; ,\\
	\ord(\til{\mu}A):\quad 
	& (i\omega -J) \psi_{101}  = - a\ev+ D^{20}F(\ev, \psi_{001}) + D^{11}F\cdot \ev \; . \label{eq_a}
\end{align}
Equation  \eqref{eq_a} is solvable if its right hand side is orthogonal to the kernel of the adjoint operator $-i\omega - J^\ast$. Thus we conclude 
\begin{equation}\label{formula_a}
a\prod{\ev}{\ev^\ast} = \prod{D^{20}F(\ev, \psi_{001}) + D^{11}F \ev}{\ev^\ast},
\end{equation}
where $\prod{\cdot}{\cdot}$ denotes the Hermitian scalar product, $J^*=J^T$ is the adjoint matrix (or simply transposed, since $J$ is real) to $J$ and $\ev^*$ is the adjoint eigenvector to $-i\omega$ fulfilling $J^\ast\ev^\ast=-i\omega\ev^\ast$, scaled with a complex factor to ensure  $\prod{\ev}{\ev^\ast}=1$. For more details see \citet{Haragus_2011}. Considering the fact that $J$ is invertible  and  solving the equation \eqref{eq_psi001} for $\psi_{001}$ we can calculate the parameter $a$ in the normal form using formula  \eqref{formula_a}.

Considering now higher order terms in $A$ and $\overline A$ we obtain 
\begin{align}
	\ord(A^2):
	&& (2i\omega - J)\psi_{200} &= \frac{1}{2} D^{20}F(\ev,\ev) \label{eq_psi200}\; ,\\
	\ord(|A|^2): 
	&& - J \psi_{110}& =  D^{20}F( \ev,\cc{\ev}) \label{eq_psi110}\; ,\\
	\ord(A|A|^2): 
	&& b\ev + (i\omega - J) \psi_{210} &=D^{20}F(\ev,\psi_{110}) + D^{20}F(\cc{\ev},\psi_{200})\nonumber \\
	&& &+ \frac{1}{2} D^{30}F(\ev,\ev,\cc{\ev})\; . \label{eq_b}
\end{align}
The solvability condition for equation  \eqref{eq_b} yields 
\begin{equation}\label{formula_b}
b\prod{\ev}{\ev^*}=\prod{D^{20}F(\ev,\psi_{110}) + D^{20}F(\cc{\ev},\psi_{200}) + \frac{1}{2} D^{30}F(\ev,\ev,\cc{\ev})}{\ev^*}\; .
\end{equation}
Since $2i\omega$ and $0$ are not eigenvalues of $J$, we can solve equations \eqref{eq_psi200} and \eqref{eq_psi110} for $\psi_{200}$ and $\psi_{110}$, respectively, and determine  the constant $b$ by formula \eqref{formula_b}. 

Due to the normal form theory for dynamical systems, the sign of the real part of $b$ determines the type of the Hopf bifurcation and stability of the emerging periodic solutions.
If $\Re b < 0 $, stable periodic orbits arise to the side of the bifurcation where the steady state becomes unstable. This is called a \emph{supercritical} Hopf bifurcation. A \emph{subcritical} Hopf bifurcation occurs if the real part of $b$ is positive and unstable orbits arise to the side of the bifurcation where the steady state is stable.

Applying standard numerical methods like Newton's and continuation methods used to compute stationary solutions and eigenvalues of the Jacobian matrix in Sect.~\ref{sec:local}, as well as LU decomposition  \citep{nr3} for the matrices in \eqref{eq_psi200}, \eqref{eq_psi110} and \eqref{formula_b},  we obtain $b$  for the critical  parameter values $\mu_1$ and $\mu_2$.  For both critical values of the parameter, $\Re b$ is negative (Table \ref{tab:critical}). Consequently,  both Hopf bifurcations are supercritical and the families of periodic orbits are stable. The sign of the derivative $\frac{d}{d\mu} \Re \lambda_{1} (\mu_i)$, $i=1,2$ determines the direction in which the orbits appear. In the first critical point they arise to the right ($\mu_1<\mu$), while at the second they emerge to the left ($\mu<\mu_2$), as clearly seen in Fig.~\ref{fig:orbits}.

\end{proof}

\section{Global continuation of a family of periodic solutions} \label{sec:global}

While the results in Sect.~\ref{sec:local} only give the existence of stable periodic orbits in a small neighborhood of the critical  values of $\mu$, we show the existence of a continuous path of periodic orbits between those two bifurcation points. For this purpose we apply Theorem A in \citet{Alexander_1978} and Theorem 4.1, Propositions~5.1, 5.2 and  Snake Termination Principle in \citet{Mallet_1982}. In more detail those theorems state that: A  family of periodic orbits bifurcating from the stationary solution at a Hopf point can be continuously extended
and become either unbounded with respect to  parameter $\mu$, period $T$, or solution $u$, or converge to another Hopf point. 

Following the notation in \citet{Mallet_1982}, the center index \zhong can be calculated to distinguish so-called \emph{sources} ($\textnormal{\zhong}=+1$) from \emph{sinks} ($\textnormal{\zhong}=-1$) of the path and to give a direction to the path of orbits connecting two Hopf points. The center index is defined as 
\[
\textnormal{\zhong}(u^\ast(\mu^\ast),\mu^\ast)=\chi (-1)^{E(\mu^*)}\; ,
\]
where $E(\mu) \in \mathbb{N}$ denotes the sum of the multiplicities of the eigenvalues of the Jacobian matrix $J(u^*(\mu),\mu)$ having strictly positive real parts. The number $\chi \in \mathbb{Z}$ is the crossing number denoting the net number of pairs of eigenvalues crossing the imaginary axis at $\mu^\ast$, defined by
\[
 \chi = \frac{1}{2}(E(\mu^\ast+)-E(\mu^\ast-))\; ,
\]
where $E(\mu^\ast+)$ and $E(\mu^\ast-)$ denote right- and left-hand limits of $E$ at $\mu^\ast$.

We shall say that $(u^\ast(\mu^\ast), \mu^\ast)$ is a Hopf point if the condition of the Hopf bifurcation theorem are satisfied. 

\begin{theorem}
The model \eqref{eq:model} possesses a continuous path of periodic solutions connecting two Hopf bifurcation points, where the first Hopf point $(u^\ast(\mu_1), \mu_1)$ is a source and the second Hopf point $(u^\ast(\mu_2), \mu_2)$ is a sink for the path.
\end{theorem}

\begin{proof}
We show in the following that for the model \eqref{eq:model} neither parameter nor period nor solution can become unbounded, so the path of periodic orbits bifurcating from the stationary solution must end in another Hopf point. Theorem~\ref{existence} ensures uniform boundedness of solutions of the model \eqref{eq:model}.  The parameter $\mu \in \mathcal M$ is naturally bounded, since zinc concentrations in the soil only appear within a certain range. The least period $T$ of the orbits is a continuous function of $\mu$. Numerical computations show that the period is also bounded for all $\mu \in [\mu_1, \mu_2]$, see Fig.~\ref{fig:perifloq}~(a).

Thus, since parameter, period and solution remain bounded, and the domain $\mathcal M$ of the parameter contains exactly two  bifurcation points,  the path of periodic orbits emerging from one Hopf point must terminate in another Hopf point, see \citet[Propositions~5.1, 5.2]{Mallet_1982}. Therefore,  there exists a continuous family of periodic solutions  between those two Hopf points. An illustration of this path of periodic orbits as a two dimensional projection is given in Fig.~\ref{fig:orbits}.

In our case, exactly two eigenvalues cross the imaginary axis in the critical values $\mu_{1,2}$. At $\mu_{1}$  the two eigenvalues   cross the imaginary axis from left to right, whereas   at  $\mu_2$ they cross the imaginary axis from right to left.  So we obtain $\chi=1$ for $\mu_1$ and  $\chi=-1$ for $\mu_2$.
Then   with $E(\mu_{1,2})=0$ we find $\textnormal{\zhong}(u^\ast(\mu_1), \mu_1)=+1$, so the first Hopf point (left point in Fig.~\ref{fig:orbits}) is a source of the path of periodic orbits. For the second Hopf point (right point in Fig.~\ref{fig:orbits}) we obtain  $\textnormal{\zhong}(u^\ast(\mu_2),  \mu_2) = -1$, showing that it is a sink. This result is in concord with the snake termination principle by \citet{Mallet_1980} stating that if a path of  orbits  emerging from one Hopf point terminates in a second Hopf point, then the two Hopf points must have center indexes of opposite sign.
\end{proof}

\begin{figure}
 \includegraphics[width=\textwidth]{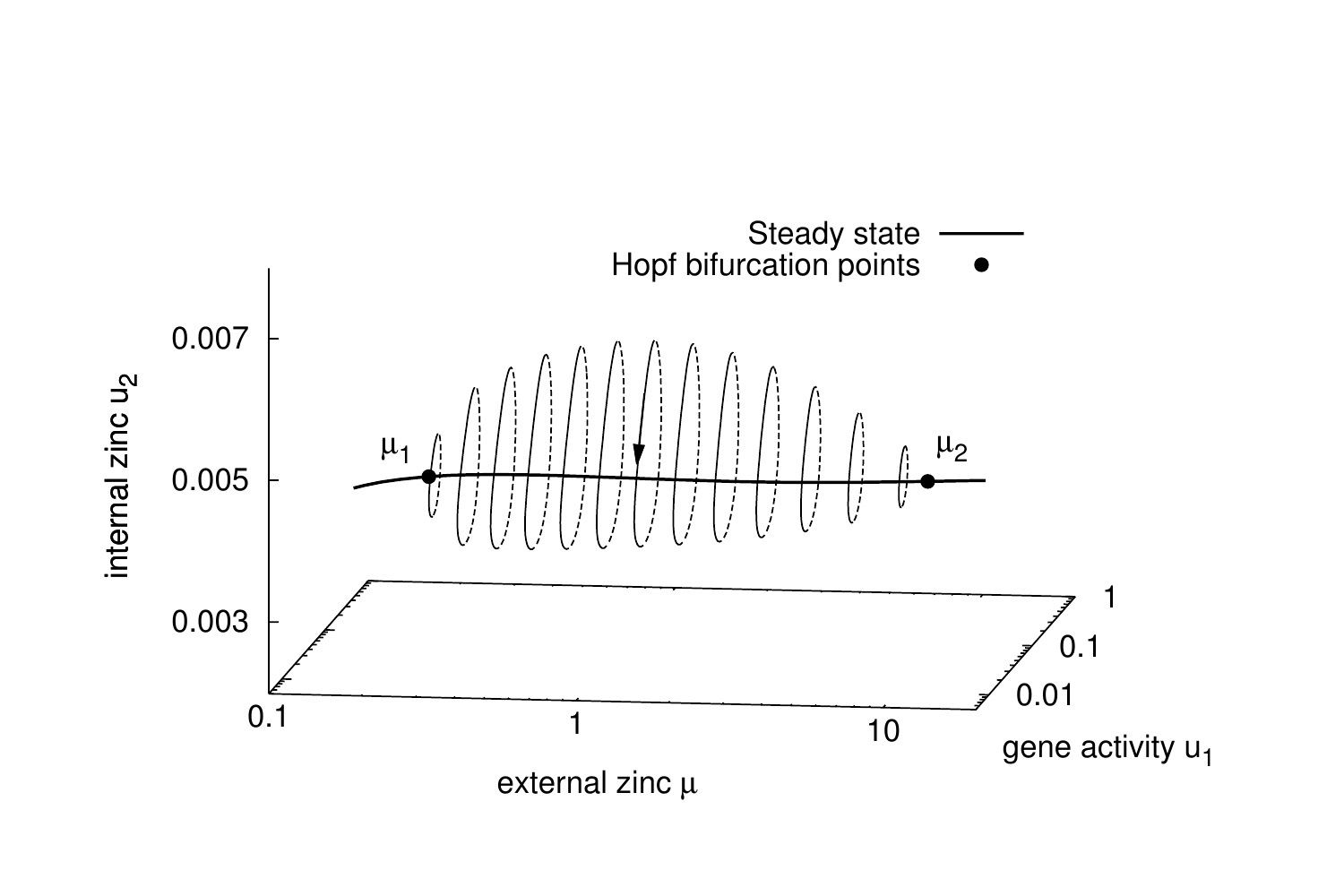}
 \caption{Illustration of the path of periodic orbits between the two Hopf bifurcation points $(u^\ast(\mu_1), \mu_1)$ and $(u^\ast(\mu_2), \mu_2)$. From the four-dimensional system only two dimensions (gene activity and internal zinc concentration) are shown. The thick solid line marks the steady state, thin lines show the stable periodic solutions. The rotation direction is indicated with an arrow.}
 \label{fig:orbits}
\end{figure}

\section{Floquet multipliers and stability of periodic orbits}
\label{sec:floquet}

Though direct numerical simulations suggest stability of the periodic orbits between the two Hopf bifurcation points (cf. Fig.~\ref{fig:orbits}), there may still be slow drifts or spirals around the orbits. Therefore, the stability of the periodic solutions is now analyzed in more detail by studying its Floquet multipliers. These multipliers can be derived as eigenvalues of the linearized Poincar\'e map and equivalently as eigenvalues of the monodromy matrix. In the following, we will focus on the latter.

\begin{theorem}
The periodic solutions of the model \eqref{eq:model} for $\mu \in [\mu_1, \mu_2]$ are asymptotically stable. 
\end{theorem}

\begin{proof}
In the previous section we have shown that 
 for a given $\mu \in  [\mu_1, \mu_2]$ the system of differential equations \eqref{eq:model} has a periodic solution $u(t)=u(\mu,t)$ with a  period  $T=T(\mu)\in (0,\infty)$. Naturally, all integer multiples of $T$ are also periods, so $T$ is chosen to be the least period. The periodic orbit is called asymptotically stable, if trajectories starting near the orbit converge to the orbit for $t\rightarrow\infty$. This stability is determined by the eigenvalues of the monodromy matrix of the system \eqref{eq:model}. For $\mu\in [\mu_1, \mu_2]$ the function $\Phi_t(u^0, \mu)$ denotes the solution $u(t)$ at time $t$ starting from $u(0)=u^0$. For a point $(u^0(\mu), \mu)$  on a periodic orbit  of \eqref{eq:model}  and  the least period of this orbit $T=T(\mu)$,  the monodromy matrix is given by
\begin{equation*}
	M=\frac{\partial \Phi_T(u, \mu)}{\partial u}\big |_{u=u^0(\mu)} \; .
\end{equation*}
The eigenvalues of the monodromy matrix are called Floquet multipliers or characteristic multipliers. The Floquet multipliers are independent of the choice of $u^0$ on the periodic orbit, whereas the monodromy matrix and its eigenvectors  depend on this choice. 
Since  for $\mu \in [\mu_1, \mu_2]$ the system \eqref{eq:model} possesses a periodic solution, one of the multipliers is always $1$ and its eigenvector points in the direction tangential to the periodic cycle, i.e. $F(u^0)$ \citep{Marx_2011}. The periodic solution is asymptotically stable, if the absolute values of all other Floquet multipliers are strictly smaller than 1 \citep{Lust_2001}.

In order to compute the Floquet multipliers numerically we first obtain $u^0(\mu)$ and $T(\mu)$ using the single shooting technique \citep{Marx_2011}. For $\mu \in [\mu_1, \mu_2]$ and all points $u^0(\mu) \in \R^4$ on a periodic orbit with period $T(\mu)$ it holds that
\begin{equation*}
	g(u^0(\mu),T(\mu)) := \Phi_T(u^0(\mu), \mu) - u^0 = 0.
\end{equation*}
Since every point on a given orbit fulfills this equation, the system is underdetermined and an additional scalar ``phase condition" of the form
\begin{equation*}
	h(u^0(\mu),T(\mu)) = 0
\end{equation*}
is needed. Most simply, the condition fixes one of the components of  the vector $u^0(\mu)$. Here, we estimated the stable periodic orbit by long-time numerical integration of the dynamical system \eqref{eq:model} by starting slightly off the steady state $u^\ast(\mu)$. Then the state $u(\tau,\mu)$ for some large enough $\tau \in (0, \infty)$ was used to fix the component $u_1^0(\mu)$ by setting $h(u^0(\mu),T(\mu)):=u_1^0(\mu)-u_1(\tau,\mu)$. The combined system
\begin{equation*}
	g(u^0(\mu),T(\mu)) = 0, \quad h(u^0(\mu),T(\mu)) = 0
\end{equation*}
was solved for $(u^0(\mu),T(\mu)) \in \R^{5}$ with a standard Newton algorithm using $u(\tau,\mu)$ and a rough estimate $T(\mu)$ as starting values. Numerical integration of \eqref{eq:model} was done with a Rosenbrock stiffly stable ODE solver \citep{nr3}. For regular periodic solutions, the Newton iteration converges locally with quadratic rate \citep[Theorem 6.25]{Marx_2011}. Having $u^0(\mu)$ and $T(\mu)$, the derivatives in the monodromy matrix were obtained by the finite difference formula and eigenvalues of the monodromy matrix  were calculated using Hessenberg form and QR iteration as in Sect.~\ref{sec:local}.

\begin{figure}
 \includegraphics[width=\textwidth]{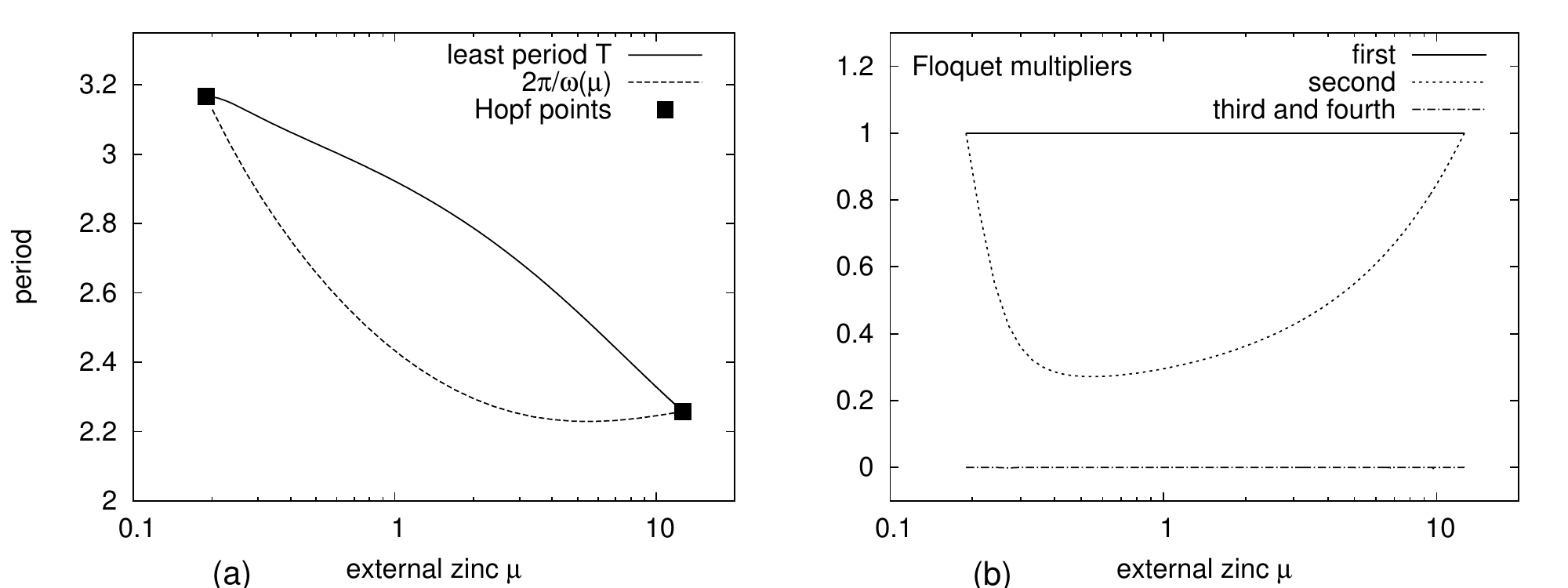}
\caption{Least period $T(\mu)$ (a) and Floquet multipliers (b) of periodic orbits as a function of the external zinc concentration $\mu$ between the two Hopf bifurcation points. (a) The numerically computed least period $T$ (solid line) is shown together with the hypothetical period $2\pi/\omega$ of the linearized system in the steady state (dashed line), where $\omega$ is the imaginary part of the first eigenvalue of the Jacobian matrix $J$. At the critical bifurcation parameter values  ($\mu_1\approx 0.19$ and $\mu_2\approx 12.6$) the lines cross, showing $\lim\limits_{\mu\rightarrow\mu^*}T=2\pi/\omega$. (b) Floquet multipliers were computed numerically. While the first Floquet multiplier equals one, the others are smaller than one. Therefore, the periodic solutions are of type 0 and asymptotically stable for all external zinc concentrations within this range.}
\label{fig:perifloq}
\end{figure}

The resulting minimal periods are shown in Fig.~\ref{fig:perifloq}(a), while the Floquet multipliers for varying external zinc concentration $\mu \in [\mu_1, \mu_2]$ are shown in Fig.~\ref{fig:perifloq}(b). 
 As expected from the theory, the first multiplier in the tangential direction of the periodic orbit equals one, while the others are real, positive, and smaller than one. Therefore, the periodic orbits are asymptotically stable for all external zinc concentrations between the two Hopf bifurcation points  $(u^\ast(\mu_1), \mu_1)$ and 
 $(u^\ast(\mu_2), \mu_2)$. In the bifurcation points, a second Floquet multiplier tends to one and information on stability is provided by local properties at the bifurcation points (see Sect. \ref{sec:local}). For $\mu \in (\mu_1, \mu_2)$ all of the multipliers besides the first are smaller than one in absolute value, no secondary bifurcation takes place, and all the periodic orbits are of type $0$ \citep{Mallet_1982}. It is also noteworthy that two of the Floquet multipliers are very close to zero. This shows that the periodic solutions are strongly attracting in two directions and corresponds  to the two large negative eigenvalues of the Jacobian matrix $J$ of the system \eqref{eq:model}.

\end{proof}

\section{Effects of buffering}\label{buffering}

In a recent paper \citep{Claus_2013} we showed that stability of the steady state is  affected by buffering. It changes the dynamical features of the system and its stability without changing the actual steady state values. Here, we want to discuss this in more detail.

Vacuolar sequestration and complexation by peptides are essential to detoxification of heavy metals in plants, \citep{Hall_2002}. Though in a long term some of these processes are irreversible, we focus here on fast and reversible binding of free zinc (buffering). We do not differ between compartmental sequestration, such as vacuolar storage, and complexation with peptides. Buffering is modeled by a simple reversible reaction with a chelator/buffer $\mathrm{C}$
\[
 \mathrm{Zn^{2+}}\ +\ \mathrm{C} \quad\leftrightarrow\quad  \mathrm{Zn C} \, .
\]
The buffer $\mathrm{C}$ is assumed to be present in high excess with practically constant concentration, and the complex $\mathrm{ZnC}$ is assumed not to be involved in transport and regulation. For more details see \citep{Claus_2013}.

The system \eqref{eq:model} needs to be extended by the aforementioned buffering reaction. In terms of the differential equations, a new equation for buffered zinc considering binding and unbinding to the chelator is needed
\begin{equation}\label{eq:buffer_1}
 \frac{du_5}{dt} = p_2 (p_1 u_2 - u_5) \; ,
\end{equation}
where the concentration of the buffered zinc $\mathrm{Zn C}$ is denoted as $u_5$. The equilibrium constant $p_1>0$ gives the steady state ratio of buffered to unbuffered zinc concentrations, i.e. $p_1 = \frac{u_5^\ast}{u_2^\ast}$. The  constant $p_2>0$ is the characteristic rate of the buffering reactions in relation to the other reactions. Both parameters will be denoted as $p:=(p_1,p_2)^T$ in the following. The equation for the internal zinc concentration $u_2$ needs to be extended with the same terms as Eq. \eqref{eq:buffer_1} with opposite sign to balance buffering and guarantee mass conservation
\begin{equation}\label{eq:buffer_2}
 \frac{du_2}{dt} = u_1 f(\mu) - u_2 - p_2 (p_1 u_2 - u_5) \; ,
\end{equation}
while the other equations in \eqref{eq:model} remain unchanged. The first four components of the stationary solution, i.e.   $u_1^\ast, \ldots, u_4^\ast$, are unaffected by the extension and remain the same as for system \eqref{eq:model}. The steady state buffered zinc concentration is simply proportional to the internal zinc concentration ($u_5^\ast = p_1 u_2^\ast$). The dynamic behavior of the system, however, changes dramatically. We want to analyze this change with respect to the buffering parameters $p$.

Most informative for the dynamic behavior of the system around the steady state is the largest eigenvalue of the Jacobian matrix. The Jacobian of the extended system with buffering evaluated at the steady state  is given by
\begin{equation}
\label{eq:jacbuff}
   J_b(u^\ast(\mu,p),\mu, p) =- \begin{pmatrix}
       \phantom{+}a+1 & \phantom{+}0 & \ -b & \ \phantom{+}0 & \ \phantom{+}0 \\
       -c & \ p_1 p_2+1 & \ \phantom{+}0 & \ \phantom{+}0 & \ -p_2\\
       \phantom{+}0 & \phantom{+}0 & \ \phantom{+}d+1 & \ \phantom{+}e & \ \phantom{+}0 \\
       \phantom{+}0 & -f & \ \phantom{+}g & \ \phantom{+}h+1 & \ \phantom{+}0 \\
       \phantom{+}0 & -p_1 p_2 & \ \phantom{+}0 & \ \phantom{+}0 & \ \phantom{+}p_2
      \end{pmatrix}
\end{equation}
where positive numbers $a,b,c,d,e,f,g,h$ are introduced for simplification:
\begin{equation*}
\begin{array}{ll}
  a=\kappa (u^\ast_3)^2\; ,  \quad & b=2\kappa u^\ast_3(1-u^\ast_1)\; ,\\
  c=f(\mu)\; ,  &d=\gamma_1 u^\ast_4\; ,\\
  e=\gamma_1 u^\ast_3\; ,  &f=\gamma_2 (1-u^\ast_4)\; ,\\
  g=\gamma_3 u^\ast_4\; ,  &h=\gamma_3 u^\ast_3 + \gamma_2 u^\ast_2\; .
\end{array}
\end{equation*}
These values appeared already in the Jacobian matrix $J$ of the unbuffered system [compare \eqref{eq:jac}].
The characteristic polynomial $\chi(\lambda, p)$ of $J_b$ is
\begin{equation}\label{eq:charpol}
 \begin{split}
 \chi(\lambda, p) = & - bcef (\lambda + p_2) - (a+1+\lambda)(-eg + (d+1+\lambda)(h+1+\lambda)) \times \\
  & (\lambda^2 + p_2 + \lambda (1+p_2 + p_1 p_2)).
 \end{split}
\end{equation}
For an eigenvalue $\lambda(\mu, p)$ of $J_b$ it holds that $\chi(\lambda(\mu, p),p) \equiv 0$ and $\frac{d \chi}{d p_i} \equiv 0$, $i=1,2$. Having the solution $\lambda_i^0$ of $\chi(\lambda(\mu, p),p) \equiv 0$  for $p_i=0$ (no buffering), $i=1,2$, a continuation method can be applied
\begin{equation}\label{eq:charpolconti}
\begin{aligned}  
  \frac{\partial \lambda}{\partial p_i} &= -\frac{\partial \chi}{\partial p_i}/ \frac{\partial \chi}{\partial \lambda} \; , \quad p _i \in [0,P_i]\; ,\\
  &\lambda(\mu,0)=\lambda_i^0 \; ,
\end{aligned}
\end{equation}
with $P_i>0$, $i=1,2$ and $\mu \in \mathcal{M}$. This equation was solved numerically using a Dormand-Prince method with adaptive step size for solving ordinary differential equations \citep{nr3}. 

\begin{figure}
 \includegraphics[width=\textwidth]{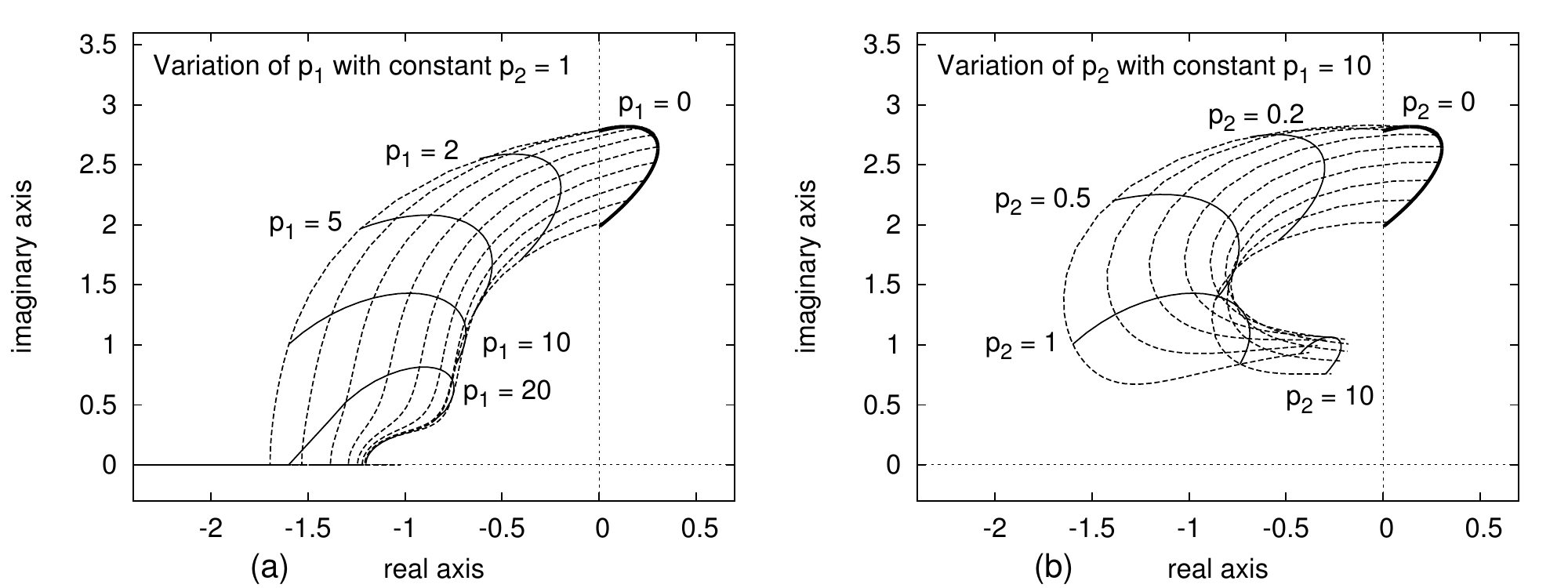}
\caption{Largest eigenvalue $\lambda_1$ of the Jacobian matrix  $J_b$ of the buffered system for variations of the two buffering parameters (a) $p_1$ and (b) $p_2$. Dashed lines show isolines for constant external zinc concentration $\mu$ in the originally unstable region between $0.18$ and $12.6$. Solid lines mark isolines for constant $p_1$ (a) or $p_2$ (b), respectively. Increasing either $p_1$ (a) or $p_2$ (b) stabilizes the system by decreasing real and imaginary part of the eigenvalue.}
\label{fig:buffer}
\end{figure}

The effect of one of the buffering parameters can be calculated by choosing the other to be constant and by solving $\chi(\lambda(\mu, p),p) \equiv 0$ via \eqref{eq:charpolconti} for a given $\mu$. 
Fig.~\ref{fig:buffer} presents the result for both parameters and for various external zinc concentrations $\mu \in \mathcal M$. Both an increase in $p_1$ and an increase in $p_2$ decreases the real part of the two complex eigenvalues below zero and stabilizes the steady state for all external zinc concentrations. In biological terms this means that buffering needs to be sufficiently strong (high $p_1$) and sufficiently fast (high $p_2$) to stabilize the system. Equilibrium constants $p_1$ of roughly over $1.1$ ensure asymptotic stability of the stationary solution (Fig.~\ref{fig:buffer}a). Indeed, experimental measurements suggest that most zinc in the cells is bound to chelators and buffers \citep{Dittmer_2009, Vinkenborg_2009}, so $p_1$ can be estimated to be in the range of at least $1000$. Under such high values of $p_1$ all eigenvalues are real and negative,  and 
oscillations do not occur. The constant $p_2$ does not seem to play such an important role. The eigenvalues turn negative already at $p_2\approx 0.09$. Increasing $p_2$ beyond $1$ results even in less damping (Fig.~\ref{fig:buffer}b), where the eigenvalue eventually converges without crossing the imaginary axis when $p_2 \to \infty$.\\

So far, we analyzed how buffering changes the stability of the system for different constant values of $\mu$. However, we have not dealt with the reaction to a $\mu(t)$ that varies in time. This parameter is the external zinc concentration, which tends to vary in time under natural soil conditions. To get an idea of how dynamic variation of $\mu$ affects the system, assume that the system stays close to the steady state while it varies. Further, assume that variation in $\mu$ is small around a given constant $\mu^0$. Under these conditions, an about the steady state and $\mu^0$ linearized system reflects the behavior of the nonlinear problem sufficiently well and the concept of the transfer function of linear time invariant systems can be applied. Within this concept the model \eqref{eq:model} of zinc homeostasis is considered to be a control system that maps an external (\emph{input}) to an internal concentration (\emph{output}). In more detail, the input is here the relative variation in external zinc $x := (\mu - \mu^0)/\mu^0$, and the output is the relative variation in internal zinc $y := (u_2 - u_2^0)/u_2^0$, where the steady state $u^\ast$ for $\mu=\mu^0$ was denoted with $u^0 := u^\ast(\mu^0)$. Note that the input and output were scaled to account for the difference in magnitude of $u_2$ and $\mu$ ($\sim 0.01$ against $\sim 1$). 

Linearization about $u^0$ and $\mu^0$ delivers
\begin{equation}\label{eq:lti}
\begin{aligned}
 \frac{d \tilde u}{d t} & = J_b^0\, \tilde u + B^0 x\; ,\\
 y &=C^0 \cdot \tilde u \; ,
\end{aligned}
\end{equation}
where $\tilde u = u - u^0$, $J_b^0=J_b(u^0,\mu^0)$, $B^0= \mu^0\,\frac{\partial F}{\partial \mu}(u^0,\mu^0) = \mu^0 u_1^0\partial_\mu f(\mu^0)\, \eta$, $\eta=(0,1,0,0,0)^T$ and $C^0=\eta/u_2^0$. Now we determine how a time dependent input $x(t)$ is mapped to an output $y(t)$. For this purpose \eqref{eq:lti} is Laplace transformed
\[
 \mathcal{L}y  = C^0 \cdot \left(\left(sI - J_b^0\right)^{-1} B^0 \right) \mathcal{L}x\; ,
\]
where $\mathcal{L}y$ and $\mathcal{L}x$ are the Laplace transforms of $y$ and $x$, respectively, and $s = \sigma +i \omega \in \mathbb{C}$ with $\sigma, \omega>0$. The Laplace transforms of the input and output are linearly related by the \emph{transfer function} $G(s):=C^0 \cdot \left(\left(sI - J_b^0\right)^{-1} B^0 \right)$. In principle the system is considered to be a driven oscillator and the transfer function tells how strong an input oscillating with frequency $\omega$ is transferred to the output signal. More elaborate input signals are decomposed by a Fourier or Laplace transform into harmonics of different frequencies allowing to estimate the output solely with the transfer function. We determined $G$ numerically for $\mu^0=1$, $p_2=1$ and various values of $p_1$ and present it in a Bode plot in Fig.~\ref{fig:bode}. The real part of $s$ is usually set to zero and not shown in Bode plots.

The details of how Fig.~\ref{fig:bode} was obtained follow. Firstly, the steady state of \eqref{eq:model} extended by the equations \eqref{eq:buffer_1} and \eqref{eq:buffer_2} for the given $\mu^0$ was determined. Secondly, $J_b^0$, $B^0$ and $C^0$ were calculated for $p_2=1$ and various $p_1$. Finally, $G$ was calculated using GNU Octave's \emph{Computer-Aided Control System Design} toolbox.
\begin{figure}
 \includegraphics{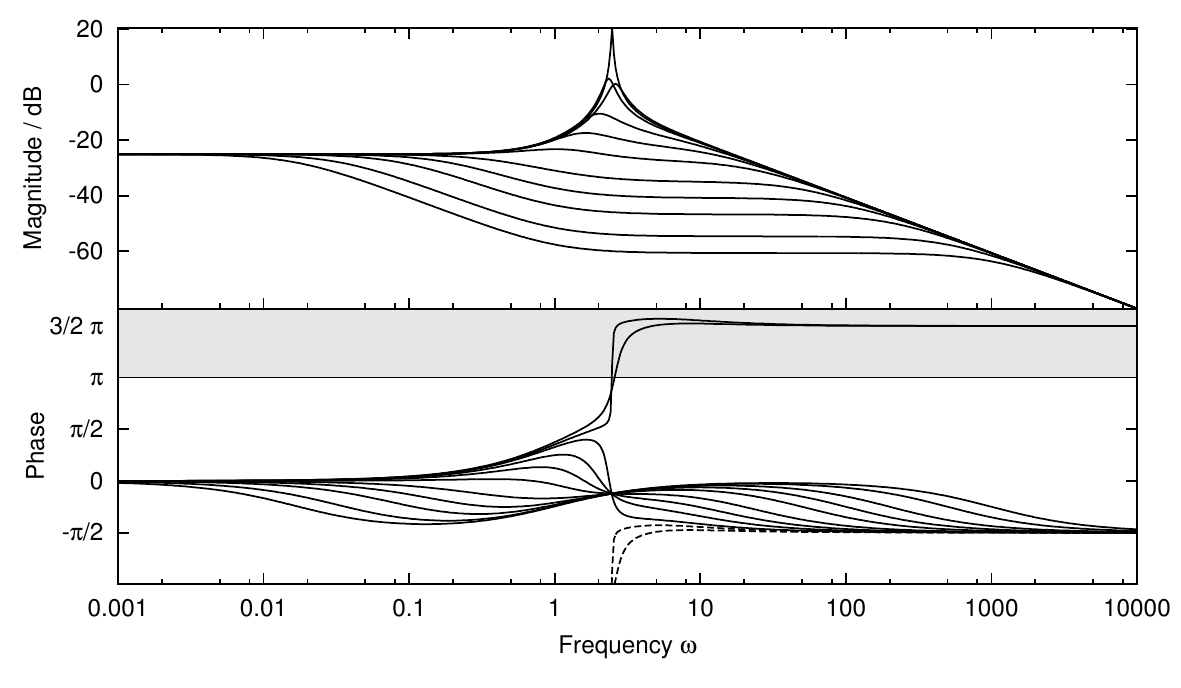}
\caption{Bode plot of the transfer function $G(s)$ of the system with buffering for $\mu^\ast=1$, $p_2=1$ and various $p_1$ between $0$ and $1000$. For small $p_1$ below $\sim 1.1$ (unstable steady state) a clear resonance appears. The accumulated phases of the unstable cases go beyond $\pi$ (grey box) and the corresponding principal values $\operatorname{Arg} G(s)$ are shown as dashed lines.}
\label{fig:bode}
\end{figure}
The gain of the system is in general small and below $-20\, \mathrm{dB}$, appart from a resonance for the cases where the steady state has a low stability or is unstable (Fig.~\ref{fig:bode}). Considering that the system is responsible for homeostasis of zinc, it makes sense that the gain is small to decouple the internal concentration as much as possible from the external concentration. A resonance at the frequency of the periodic orbits is clearly present for $p_1$ roughly smaller than $1.1$, as expected from the results on the global Hopf bifurcation in  Sect.~\ref{sec:global}. The accumulated phases of these unstable cases exceed $\pi$ and tend to $\frac{3}{2}\pi$ for $\omega \to \infty$. The corresponding principal values $\operatorname{Arg} G$ are shown as dashed lines. Increasing $p_1$ damps the system decreasing the tendency to resonate. This is seen by the sudden change in the phase response, where the argument of $G$ stays smaller than $\pi$ for all $\omega$ (Fig.~\ref{fig:bode} bottom). We found previously in Fig.~\ref{fig:buffer}a that very high values of $p_1$ overdamp the system. Here this behavior is observed by a substantial decrease in gain (ca. $-60\, \mathrm{dB}$) and an in phase oscillation for $p_1 =1000$ and frequencies from $2$ to $400$.

\section{Conclusion and biological implications}

Previous simulations showed that the zinc uptake and regulation model considered here tends to have periodic solutions for certain external zinc concentrations and parameter choices \citep{Claus_2012}. The true nature of the oscillation and the effect of buffering were not treated in that publication. Here we focused on proving the  existence and stability of the periodic solutions and analyzed the effect of buffering by extending the model. 

To reduce the complexity of the system we eliminated the purely linear equations from the original six-dimensional model and focused this analysis on four non-linear equations. For further reduction of the system's dimensionality one can look at the eigenvector belonging to the smallest eigenvalue $\lambda_4$, i.e. the direction of fastest decay, see Sect.~\ref{sec:local}. Indeed, as part of the results it was found that this eigenvector points mainly in the direction of the inhibitor $u_4$ (data not shown). Thus, it is possible to further reduce the system to three dimensions by setting $\frac{du_4}{dt}=0$ as a quasi steady state. Since the substitution of $u_4$ leads to more complicated equations and since the reduction of dimensions was not our main priority, we studied here the behavior of the four-dimensional system.

The naturally variable external zinc concentration $\mu$ was considered as the bifurcation parameter. The periodic orbits were found to bifurcate from stationary solutions $u^\ast(\mu)$ at two Hopf bifurcation points  
$(u^\ast(\mu_1), \mu_1)$ and $(u^\ast(\mu_2), \mu_2)$, 
where  $\mu_1 \approx 0.19$ and $\mu_2 \approx 12.64$,  and to form a continuous family  of stable periodic solutions. 
By  computing  the Floquet multipliers  the periodic solutions were found to be stable. Further, the transfer function of the linearized system was determined to analyze how the system reacts to dynamic variation in $\mu$ and how buffering influences this reaction. In an ideal homeostatic system the gain should be very low for all conditions to decouple the input and output. Indeed, we find a generally low gain---up to a resonance corresponding to the frequency of the periodic solutions for low buffering. The gain is in particular low for high frequencies supporting the homeostatic property of the system.

From a biological point of view, it is assumed to be optimal for plants to keep an almost constant zinc concentration under wide ranges of external zinc supply. In this model, a high value of $\gamma_1$, the binding affinity between activator and inhibitor, is required to make the steady state insensitive to variations in $\mu$. Without buffering this high affinity, however, leads to a change in the system's dynamic behavior resulting in instability of the steady state and oscillations. Such oscillations may generate toxic zinc peaks in the cells and therefore pose a dangerous threat for the plant. Without considering the important effect of buffering, we proposed that $\gamma_1$ should be smaller, yielding a model with less robust but stable steady states. Later results in \citet{Claus_2013} and in this paper showed that stability can also be affected by buffering without changing the steady state. We found that already weak buffering can switch the dynamic behavior of the system from oscillations to a stable steady 
state 
and thus protect the plant cells against toxic zinc shocks. Thus, our results suggest that buffering is not only important as a mechanism against fast transient supply changes, but also to stabilize the regulatory system and prevent strong self-oscillatory behavior.

\appendix
\section{Appendix: proof of existence and uniqueness of global solution}\label{appendix:exitence}
\begin{proof}[Theorem \ref{existence}]
Due to the  local Lipschitz continuity of $F$,  we obtain the local in time existence and uniqueness of a solution $u(t,\mu)$ of \eqref{eq:model} by the Picard-Lindel\"of theorem.

To prove the uniform boundedness of the solution, we show that  $\mathcal S$ is a positively invariant region for the system \eqref{eq:model}, i.e. that any trajectory $u(t,\mu)$ with $u(0,\mu)\in \mathcal S$ remains in $\mathcal S$ for all times $t\geq0$.  From the equations of the system  \eqref{eq:model} and by using the positivity of the coefficients we obtain the following estimates
\begin{align*}
F_1(u)|_{u_1=0}&=\kappa u_3^2 \geq 0, &\quad F_2(u)|_{u_2=0} &= u_1f(\mu) \geq 0 &\quad \text{ for }  u_1 \geq 0 \; , \\
F_3(u)|_{u_3=0}& = 1 > 0 , &\quad F_4(u)|_{u_4=0} &= \gamma_2 u_2\geq 0 &\quad  \text{ for }  u_2 \geq 0 \; ,
\end{align*} 
which by applying the  invariant region theorem \citep[Theorem 14.7 in][Theorem 16.9]{Smoller,Amann} with $G_i(u) = -u_i$, for $i=1,2,3,4$, imply the lower bound $u\geq 0$. To show the upper bound we use the fact that $f(\mu) < 1$ and obtain
\begin{align*}
F_{1}(u)|_{u_1=1}& = -1 \leq 0, &    F_{2}(u)|_{u_2=1} &= u_1f(\mu)-1 \leq 0  \text{ for } u_1 \leq 1 \;, \\
F_{3}(u)|_{u_3=1} &= -\gamma_1 u_4 \leq  0 \text{ for } u_4\geq 0, \hspace{-0.1 cm } &   F_4(u)|_{u_4=1} &=-\gamma_3 u_3 - 1 \leq 0   \text{ for } u_3\geq 0 \; .
\end{align*}
Then with $G_i(u)= u_i -1$ the invariant region theorem ensures that $u_i \leq 1$, for $i=1,2,3,4$. 
Thus, any trajectory starting in $\mathcal S$ remains bounded within this set. The  global existence and uniqueness of a solution  of  \eqref{eq:model} is then implied by the uniform boundedness of $u(t,\mu)$ together with continuous differentiability of $F$. The smoothness  of  $F:\mathbb R^4 \times \mathcal M \to \mathbb R^4$  with respect to $u$ and $\mu$ ensures also the smoothness of  the solutions $u$ of the system \eqref{eq:model} \citep{Amann}.
\end{proof}

\section{Appendix: proof of existence of unique and positive steady state}\label{appendix:ss}
The proof of Theorem \ref{ssunique} will be based on the following two Lemmata. In these, $f:=\frac{\mu}{\mu+ K}$ will be used as parameter instead of $\mu$ and $K$, as for any $K\in (0,\infty)$ the function $f(\mu)$ is bijective in $[0,\infty)$ and every $f\in\left(0,1\right)$ is uniquely identified with a $\mu \in (0,\infty)$.
We assume $\kappa,\,\gamma_1,\,\gamma_2,\,\gamma_3 \in\left(0,\infty\right)$
and denote the parameter space by $\mathcal{P}:=\left(0,1\right)\times\left(0,\infty\right)^4$ with elements $p=\left(f,\,\kappa,\,\gamma_1,\,\gamma_2,\,\gamma_3\right)\in\mathcal{P}$.

\begin{mylemma}
\label{prp_ss_a_priori}\label{prp_ss_exist}
For any $p\in\mathcal{P}$ the system \eqref{eq:model} has at least one steady state $u^\ast\in \interior \mathcal{S}$ and none on $\partial \mathcal{S}$. Further, any steady state $u^\ast \in \interior\mathcal{S}$ corresponds to a root $u^\ast_3$ in $\mathcal{S}_3:=\left(\frac{1}{1+\gamma_1},1\right)$ of the polynomial
\begin{equation}\label{eq:phifunction}
\begin{split}
\phi\left(x\right)
		&:=
		\gamma_3 \kappa x^4
		+ \left(f \gamma_2  \left(1+\gamma_1\right) + 1 - \gamma_3 \right) \kappa x^3 \\
		& \quad\ +\left(\gamma_3 - \kappa - f \gamma_2 \kappa\right) x^2
		+ \left( 1 - \gamma_3 \right) x
		- 1 \; ,
\end{split}
\end{equation}
and any root of $\phi$ in $\mathcal{S}_3$ to precisely one steady state in $\interior \mathcal{S}$.
\end{mylemma}
\begin{proof}
Existence of $u^\ast$ follows from Brouwer's fixed point theorem \citep[cf. Theorem\ 12.10]{Smoller},
since $\mathcal{S}$ is a positively invariant (see Theorem \ref{existence}), compact and convex subset of $\mathbb{R}^4$.

A steady state $u^\ast$ in $\mathcal{S}$ is a solution of $F(u^\ast,p)=0$. Equation $F_3(u^\ast,p)=0$ implies $u_3^\ast>0$. Using $F_1(u^\ast,p)=0$ to $F_3(u^\ast,p)=0$, expressions for 
$u^\ast_1$, $u^\ast_2$, and $u^\ast_4$ as functions of $u_3^\ast$ are found:
\begin{equation}\label{eq_ss_u124_by_u3}
  \begin{aligned}
  u_1^\ast&= \frac{\kappa u_3^{\ast 2}}{1 + \kappa u_3^{\ast 2}}\; ,& 
  u_2^\ast&= \frac{f \kappa u_3^{\ast 2}}{1 + \kappa u_3^{\ast 2}}\; , &
  u_4^\ast&= \frac{1-u_3^\ast}{\gamma_1 u_3^\ast}\; .
  \end{aligned}
\end{equation}
Note that $u_1^\ast$ and $u_2^\ast$ increase whereas $u_4^\ast$ decreases strictly mono\-tonously with $u_3^\ast$.
In particular we have $u_2^\ast>0$, as $f>0$ and $u_3^\ast > 0$. Assuming $u_4^\ast=0$ or $u_4^\ast = 1$ leads to a contradiction in $F_4(u^\ast,p)=0$, and hence, to $u_4^\ast\in\left(0,1\right)$. The last equation in \eqref{eq_ss_u124_by_u3} delivers a monotonic expression for $u_3^\ast$ in dependence of $u_4^\ast$. From monotonicity it follows that $u_3^\ast$ as a function of $u_4^\ast$ assumes its extrema at the bounds of $(0,1)$, and thus, $u_3^\ast\in\mathcal{S}_3$. Correspondingly, by monotonicity of $u_1^\ast$ and $u_2^\ast$ with respect to $u_3^\ast$, we find $u_1^\ast, u_2^\ast \in (0,1)$. Thus, $u^\ast \in \interior\mathcal{S}$ and $u^\ast \notin \partial S$ for any $p \in \mathcal{P}$.

Substituting \eqref{eq_ss_u124_by_u3} into $F_4(u^\ast,p)=0$ yields the equation $\phi\left(u_3^\ast\right)=0$. Thus, any steady state $u^\ast \in \interior \mathcal{S}$ corresponds to a root of $\phi$ in $\mathcal{S}_3$. Conversely, if $u_3^\ast$ is a root of $\phi$ in $\mathcal{S}_3$, then \eqref{eq_ss_u124_by_u3} delivers unique $u_1^\ast$, $u_2^\ast$ and $u_4^\ast$, which fulfill $F(u^\ast,p)=0$ and $u^\ast \in \interior \mathcal{S}$.
\end{proof}

\begin{mylemma}
\label{prp_ss_detjac_pos}
The number of steady states of system \eqref{eq:model} in $\mathcal{S}$ is constant for all $p \in \mathcal{P}$, and the determinant of the Jacobian \eqref{eq:jac} is strictly positive when evaluated at a steady state $u^\ast \in\mathcal{S}$.
\end{mylemma}
\begin{proof}
Consider $D:=\det J\left(u^\ast,p\right)$ where $u^\ast$ is a steady state corresponding to the parameter $p$.
The relations \eqref{eq_ss_u124_by_u3} may be used to express $u_1^\ast$, $u_2^\ast$, and $u_4^\ast$ by $u_3^\ast$, and the determinant is then
\[
\begin{split}
D
&=\frac{1}{u^{\ast}_3(1+\kappa u^{\ast\,2}_3)}\Big(\gamma_3\, \kappa^2 u^{\ast\,6}_3 + \left(\kappa f\gamma_2 + 2 \gamma_3 + \kappa\right)\kappa u^{\ast\,4}_3 + (2\kappa + \gamma_3)u^{\ast\,2}_3+ 1 \\
& \qquad\qquad\qquad\quad + \big(2 f \gamma_2 (1+\gamma_1)u^{\ast}_3 - f\gamma_2\big)\kappa u^{\ast\,2}_3 \Big) \; ,
\end{split}
\]
where $u^\ast_3 \in \mathcal{S}_3$ as defined as in Lemma \ref{prp_ss_exist}. Up to $-f\gamma_2$, all terms are positive. Using $u^\ast_3 > \frac{1}{1+\gamma_1}$ we find
\[
  2 f \gamma_2 \left(1+\gamma_1\right) u_3^\ast - f\gamma_2 > f \gamma_2 > 0\;.
\]
Thus, $D>0$ for any $p \in \mathcal{P}$ and any corresponding steady state $u^\ast \in \mathcal{S}$.

Since $\mathcal{P}$ is connected and $\mathbb{N}$ is discrete, the number of steady states in $\mathcal{S}$ could only be non-constant if it was discontinuous in $\mathcal{P}$. By Lemma~\ref{prp_ss_a_priori} a steady state $u^\ast$ cannot enter or leave the compact set $\mathcal{S}$ upon variation of $p$, because any steady $u^\ast$ depends continuously on $p$ and does not cross over $\partial\mathcal{S}$. Hence, the only way the number of steady states could change is by bifurcation of steady states (pitchfork or saddle node), which is not possible because $D>0$ in $\mathcal{P}$. In total, we find that the number of steady states in $\mathcal{S}$ has to be constant for any $p \in \mathcal{P}$.
\end{proof}

\begin{proof}[Theorem \ref{ssunique}]
By choosing one specific parameter in $\mathcal{P}$, namely  $f=\frac{1}{2}$ and 
$\kappa=\gamma_1=\gamma_2=\gamma_3=1$, we obtain
\[
\phi(x) = x^4 + x^3 -\frac{1}{2}x^2 - 1 \; ,
\]
and $\mathcal{S}_3=\left(\frac{1}{2},1\right)$. We find the existence of at least one root because $\phi(\frac{1}
{2})=-\frac{15}{16}<0$ and $\phi(1)=\frac{1}{2}>0$. The uniqueness is proven 
by strict monotonicity of $\phi$ in $\mathcal{S}_3$:
\[
 \frac{d \phi}{d x}=4\,x^3 + 3\,x^2 - x> \frac{1}{4} > 0 \; .
\]
By Lemma \ref{prp_ss_exist} this corresponds to exactly one steady state in 
$\mathcal{S}$ and from Lemma \ref{prp_ss_detjac_pos} the same holds for any 
$p\in\mathcal{P}$.

The smoothness of $u^\ast$ with respect to $p$ follows from the
smoothness of $F$ and the implicit function theorem. As $f$ is given by
a smooth function in $\mu$ (for $\kappa$, $\gamma_1$, $\gamma_2$, $\gamma_3$, $K$ fixed), the
steady state $u^\ast(\mu)$ can also be viewed as a smooth function in $\mu$.
\end{proof}

\section{Appendix: properties of the Jacobian's spectrum}\label{appendix:spectrum}
Up to permutation only the following combinations of eigenvalues $\lambda_1, \ldots, \lambda_4$  are possible for the Jacobian $J^\ast(\mu)=J(u^\ast,\mu)$ evaluated the positive steady state $u^\ast$ for $\mu \in \mathcal{M}$ and $p \in \mathcal{P}$:
\begin{enumerate}
 \item $\lambda_1, \ldots, \lambda_4 < -1$
 \item $\lambda_{1,2} < -1$ and $\lambda_3 = \overline\lambda_4$ with $\operatorname{Im} \lambda_{3,4} \neq 0$
 \item $\lambda_{1} = \overline\lambda_2$ with $\operatorname{Im} \lambda_{1,2} \neq 0$, $\operatorname{Re} \lambda_{1,2} <0$ and $\lambda_3 = \overline\lambda_4$ with $\operatorname{Im} \lambda_{3,4} \neq 0$
\end{enumerate}
In particular $J^\ast$ cannot have two pairs of conjugate eigenvalues both on the imaginary axis. In Monte Carlo simulations cases 1, 2 with $\operatorname{Re}\lambda_{3,4}$ negative, zero and positive, and 3 with $\operatorname{Re}\lambda_{3,4}<0$ have been observed. This suggests that in case 3 of two complex pairs, probably both have negative real parts.
\begin{proof}
Over $\mathbb{C}$ the Jacobian has four eigenvalues $\lambda_1,\ldots, \lambda_4$. Because of $\operatorname{det} J^\ast = \lambda_1 \lambda_2 \lambda_3 \lambda_4 > 0$ zero is not an eigenvalue and the number of negative eigenvalues is even. By $\operatorname{tr} J^\ast = \lambda_1 + \lambda_2 +\lambda_3 +\lambda_4< 0$ there must be at least one eigenvalue with strictly negative real part. The characteristic polynomial of $J^\ast$ with $\tilde\lambda := \lambda +1$ and $a,\ldots, h$ as in \eqref{eq:charpol} is
\[
 \tilde\chi(\tilde\lambda)=\left(\tilde\lambda +a\right) \tilde\lambda \left( \tilde\lambda^2 + (d+h) \tilde\lambda +(d h - e g)\right) + b c f e \; .
\]
With $d h - e g = \gamma_1 \gamma_2 u_2^\ast u_4^\ast \geq 0$ all coefficients of $\tilde\chi$ are non-negative and thus $\tilde\chi (\tilde\lambda) \geq b c fe > 0$ for $\tilde\lambda \geq 0$. Hence $J^\ast$ does not have any real eigenvalue $\lambda \geq -1$.
\end{proof}

\bibliographystyle{plainnat}
\bibliography{Claus_et_al_2014_JMB}   

\end{document}